\documentclass[reqno,draft]{amsart} 
\usepackage{amsthm,amssymb,latexsym,mathrsfs}
\usepackage{amsmath}
\usepackage{indentfirst}
\usepackage{color}
 \usepackage{cite}

\usepackage[top=2.5cm,bottom=2.5cm,left=2.5cm,right=2.5cm]{geometry}
\allowdisplaybreaks[1]
\numberwithin{equation}{section}
\newtheorem{theorem}{Theorem}[section]
\newtheorem{defn}[theorem]{Definition}
\newtheorem{lemma}[theorem]{Lemma}
\newtheorem{remark}[theorem]{Remark}
\newtheorem{prop}[theorem]{Proposition}

\allowdisplaybreaks

\begin{document}

\title[Cucker-Smale-Navier-Stokes system]{Global weak solutions to the compressible Cucker-Smale-Navier-Stokes system in a bounded domain}

\author{Li Chen}
\address{Lehrstuhl f\"{u}r Mathematik IV, Universit\"{a}t Mannheim,
Mannheim 68131, Germany}
\email{chen@math.uni-mannheim.de}

\author{Yue Li}
\address{Department of Mathematics, Nanjing University,
 Nanjing 210093, P.R. China}
\email{liyue2011008@163.com}

\author{Nicola Zamponi}
\address{Institute for Analysis and Scientific Computing, Vienna University of Technology, Wiedner Hauptstra\ss e 8-10, 1040 Vienna, Austria}
\email{nicola.zamponi@asc.tuwien.ac.at}

\begin{abstract}
A coupled kinetic-fluid model {is investigated, which describes} the dynamic behavior of an ensemble of Cucker-Smale flocking particles interacting with {a} viscous fluid in a three-dimensional bounded domain.
This system consists of {a} kinetic Cucker-Smale equation and {a} compressible Navier-Stokes system with nonhomogeneous boundary conditions. The global existence of weak solutions to this system with adiabatic coefficient $\gamma> {3}/{2}$ is established.
\end{abstract}

\maketitle


\section{ Introduction}

\subsection{The model}
Kinetic-fluid models, which describe the evolution of dispersed particles in a fluid, arise {in} many applications in industry {(}for instance, wastewater treatment \cite{BBE}, sedimentation phenomenon \cite{BBKT,S,SG}, sprays \cite{BBPS}, rainfall formation \cite{FFS} and sedimentation-consolidation processes \cite{BWC}{)}.  The Cucker-Smale kinetic model has been recently introduced in \cite{CS1,CS2,HL} to describe the flocking and swarming phenomenon of small agents. Further Cucker-Smale-fluid models are applied to investigate the same phenomenon in {fluids}, \cite{BCHK1,BCHK2,BCHK3,CHJK2,CHJK1,CL}.

Let $f(t,x,v)\geq 0$ be one-particle distribution function for flocking particles with velocity $v\in \mathbb{R}^3$
at position $x\in \Omega\subset\mathbb{R}^3$ and time $t>0$, $\rho(t,x)\geq 0$ be the density of the fluid,
and $u(t,x)\in\mathbb{R}^3$ be the velocity of the fluid.
Then, the compressible Cucker-Smale-Navier-Stokes equations read as
\begin{equation}\label{CSNS}
\left\{
\begin{aligned}
  &\partial_t f+v\cdot\nabla_x f+{\rm{div}}_v\big((u-v)f\big)-\Delta_vf+{\rm{div}}_v(fL[f])=0,\\
  &\partial_t\rho+{\rm{div}}_x(\rho u)=0,\\
  &\partial_t(\rho u)+{\rm{div}}_x(\rho u\otimes u)+\nabla_x\rho^\gamma-{\rm{div}}_x\mathbb{S}(\nabla u)=-\int_{\mathbb{R}^3}(u-v)f\,dv,
\end{aligned}
\right.
\end{equation}
where the first two terms in $\eqref{CSNS}_1$ represent free transport, and the third term in $\eqref{CSNS}_1$ describes the friction force
exerted by the fluid. Correspondingly, the fluid is influenced by the flocking particles through the force appeared on the right-hand side of
$\eqref{CSNS}_3$. The last term in $\eqref{CSNS}_1$ describes the interaction between particles who try to align with their neighbors.
The alignment operator $L$ is given by
\begin{align*}
L[f]=\int_{\Omega\times\mathbb{R}^3}K(x,y)f(t,y,w)(w-v)\,dwdy,
\end{align*}
where the kernel function $K:\Omega\times\Omega\rightarrow \mathbb{R}_+$ is symmetric.
$\rho^\gamma$ is the adiabatic pressure of the fluid, and $\mathbb{S}$ satisfies
\begin{align*}
\mathbb{S}(\nabla u)=\mu_1(\nabla_x u+\nabla_x^\top u)+\mu_2{\rm{div}}u\mathbb{I},
\end{align*}
where $\mu_1$ and $\mu_2$ are coefficients of viscosity verifying $\mu_1>0$ and $2\mu_1+3\mu_2\geq 0$, and $\mathbb{I}$ is the $3\times 3$ identity matrix.
We consider the system \eqref{CSNS} equipped with the initial data:
\begin{align}
& f(0,x,v)=f_0(x,v),\quad (x,v)\in \Omega\times\mathbb{R}^3,\label{fi}\\
& \rho(0,x)=\rho_0(x),\quad u(0,x)=u_0(x),\quad x\in \Omega,\label{ui}
\end{align}
together with the following nonhomogeneous boundary conditions:
\begin{align}
& \gamma^-f(t,x,v)=g(t,x,v),\quad (t,x,v)\in(0,T)\times\Sigma^-,\label{fb}\\
& \rho(t,x)=\rho_B(x),\quad (t,x)\in (0,T)\times\Gamma_{\rm{in}},\label{rhob}\\
& u(t,x)=u_B(x),\quad (t,x)\in (0,T)\times\partial\Omega,\label{ub}
\end{align}
where $\Omega\subset\mathbb{R}^3$ is a bounded domain with smooth boundary $\partial \Omega$,
 $\Sigma^{-}:=\{(x,v)\in\partial\Omega\times\mathbb{R}^3| v\cdot \nu(x)<0\}$,
 $\gamma ^{-}f(t,x,v)$ is the trace of $f$ on $(0,T)\times\Sigma^-$,
 and $\Gamma_{\rm{in}}:=\{x\in\partial\Omega|x\cdot\nu(x)<0\}$.
 Here $\nu(x)$ denotes the outward unit normal vector to $x\in\partial\Omega$.

\subsection{Previous results}
The mathematical analysis of the coupling of flocking kinetic equation and fluid system has received considerable {attention} in the last few years.
Below we give a short review of 
{the state of the art concerning the topic.}

We first recall some results on incompressible Cucker-Smale-Navier-Stokes systems.
Bae et al.  \cite{BCHK1} proved global existence of weak solutions and the a priori time-asymptotic
exponential flocking estimates for any smooth flow to an incompressible Cucker-Smale-Navier-Stokes system, when the viscosity is large enough.
Furthermore, they derived the well-posedness of global strong solutions to it in \cite{BCHK2}.
Choi and Lee \cite{CL} established the global existence of weak and strong
solutions to it in $\mathbb{R}^2$.
The well-posedness of weak solutions to an incompressible Cucker-Smale-Navier-Stokes system with shear thickening was given in \cite{HKKP} and \cite{MPP} by means of different methods.
The global existence of weak solutions to an Cucker-Smale-Fokker-Planck-Navier-Stokes equations in a non-perturbative
setting in $\mathbb{R}^d$ $(d=2{\;\;\rm{or}}\;3)$ and global existence of strong solutions in $\mathbb{R}^2$ were obtained in \cite{HXZ1}.
In \cite{CHJK2}, Choi et al. proved that the coupled kinetic-fluid system consisting of kinetic thermomechanical Cucker-Smale equation and Navier-Stokes system
possesses global weak and strong solutions.
In addition, they obtained a priori estimates of large-time behavior of strong solutions which exhibits an exponential alignment between TCS
particles and fluid asymptotically.

Now we focus on   compressible kinetic-fluid systems case.
In \cite{BCHK3}, Bae et al. showed the global existence of strong solutions and time-asymptotic behavior for a compressible Cucker-Smale-Navier-Stokes system.
Ha et al. established the well-posedness of classical solutions to a Cucker-Smale-Fokker-Planck-Navier-Stokes system with arbitrarily
large initial data which may contain vacuum for $1D$ and $2D$ case in \cite{HHXZ1} and \cite{HHXZ2}, respectively.
The global well-posedness of strong solutions and their emergent behavior for a  coupled system  of thermomechanical
Cucker-Smale equation and Navier-Stokes system were obtained in \cite{CHJK1}.

In case $f\equiv 0$, the system \eqref{CSNS} is reduced to the classic compressible isentropic Navier-Stokes equations. The weak solution theory of this system has been intensively developed in the literature. Lions \cite{L} obained the existence of global weak solutions to the three dimensional Navier-Stokes equations for general initial data under the restriction $\gamma\geq {9}/{5}$. Later, Feireisl, Novotn\'{y} and Petzeltov\'{a} \cite{FNP} improved the index to $\gamma>{3}/{2}$. In \cite{JZ}, Jiang and Zhang showed that the Navier-Stokes equations with axisymmetric initial data possess global axisymmetric weak solutions for $\gamma>1$. Bresch and Jabin \cite{BJ} proved the global existence of appropriate weak solutions to compressible Navier-Stokes equations with general viscous stress tensor.
In the case of nonhomogeneous Dirichlet boundary condition, the global existence of weak solutions was given by Plotnikov and Sokolowski \cite{PS} by means of Young measure and other tools. They first showed the Navier-Stokes equations possess weak solution when $\gamma$ is big enough, then this result was extended to the case of $\gamma>3/2$ with the help of kinetic theory. Lather, Chang, Jin and Novotn\'{y} \cite{CJN} gave an other proof of this result thanks to the effective viscous flux identity, oscillations defect measure and renormalization techniques for the continuity equation in the {spirit} of \cite{FNP,L}.

\subsection{Main result of this paper}

The goal of this paper is to prove global existence of weak solutions to problem \eqref{CSNS}-\eqref{ub}.
We first give the definition of weak solutions.

\begin{defn}\label{defn}
Let $T>0$ be arbitrary but fixed.
We say that a triplet $(f, \rho, u)$ is a bounded energy weak solution to problem \eqref{CSNS}-\eqref{ub} on {\rm{[}0,T\rm{]}} if it possesses the regularity
\begin{gather*}
0\leq f\in L^{\infty}(0,T;L^1\cap L^{\infty}(\Omega\times\mathbb{R}^3)),\quad|v|^2f\in L^{\infty}(0,T;L^1(\Omega\times\mathbb{R}^3)),\\
0\leq \rho\in L^\infty(0,T;L^{\gamma}(\Omega)),\quad u\in L^2(0,T;H^1(\Omega)),\quad
\rho u\in L^\infty\big(0,T;L^{\frac{2\gamma}{\gamma+1}}_w(\Omega)\big),
\end{gather*}
and fulfills the following relations:\\
1. Weak formulation of the Cucker-Smale equation: for any $\varphi\in C^2_c([0,T)\times\bar{\Omega}\times\mathbb{R}^3)$ with $\varphi=0$ on $(0,T)\times\Sigma^+$,
it holds
\begin{align*}
&\int^{T}_0\int_{\Omega\times\mathbb{R}^3}f\big(\partial_t\varphi+v\cdot\nabla_x\varphi+(u-v)\cdot\nabla_v\varphi
+L[f]\cdot\nabla_v\varphi
+\Delta_v\varphi\big)\,dxdvdt\nonumber\\
&\qquad +\int_{\Omega\times\mathbb{R}^3}f_0\varphi(0,x,v)\;dxdv
=\int^{T}_{0}\int_{\Sigma^-}\big(v\cdot \nu(x)\big)g\varphi\,d\sigma(x)\,dvdt,
\end{align*}
where $\Sigma^{+}:=\{(x,v)\in\partial\Omega\times\mathbb{R}^3|  v\cdot \nu(x)>0\}$ \emph{(}Recall  $\Sigma^{-}=\{(x,v)\in\partial\Omega\times\mathbb{R}^3|   v\cdot \nu(x)<0\}\emph{)}$;\\
2. Weak formulation of the continuity equation: for any $\psi\in C^1_c([0,T)\times\bar \Omega)$, it holds
\begin{align*}
\int^{T}_0\int_{\Omega}(\rho\partial_t\psi+\rho u\cdot\nabla_x\psi)\,dxdt
+\int_{\Omega}\rho_0\psi(0,x)\,dx=0;
\end{align*}
{3}. Weak formulation of the momentum balance equation: for any $\phi\in C^1_c([0,T)\times\bar\Omega;\mathbb{R}^3)$, it holds
\begin{align*}
\int^{T}_{0}\int_{\Omega}&\big(\rho u\cdot \partial_t\phi+\rho u\otimes u:\nabla_x\phi
+\rho^{\gamma}{\rm{div}}_x\phi-\mathbb{S}(\nabla_x u):\nabla_x\phi\\
&+(j-nu)\cdot\phi\big)\,dxdt
+\int_{\Omega}\rho_0u_0\cdot\phi(0,x)\,dx=0,
\end{align*}
where $n:=\int_{\mathbb{R}^3}f\,dv$ and $j:=\int_{\mathbb{R}^3}vf\,dv$;\\
{4}. Energy inequality: for any $\tau\in (0,T)$,
\begin{align}\label{energy main}
&\int_{\Omega}\Big(\frac{1}{2}\rho|u-u_{\infty}|^2+\frac{1}{\gamma-1}\rho^{\gamma}
+\int_{\mathbb{R}^3}\frac{|v|^2}{2}f\,dv\Big)(\tau)\,dx\nonumber\\
&+\int^{\tau}_0\int_{\Omega}\mathbb{S}(\nabla(u-u_\infty)):\nabla(u-u_\infty)\,dxdt
+\int^{\tau}_0\int_{\Sigma^-}(v\cdot \nu(x))\frac{|v|^2}{2}g\,d\sigma(x)dvdt\nonumber\\
&+\frac{1}{2}\int_0^\tau\int_{\Omega\times\mathbb{R}^3}\int_{\Omega\times\mathbb{R}^3}K(x,y)f(t,x,v)f(t,y,w)|w-v|^2\,dydwdxdvdt\nonumber\\
&+\int_0^\tau\int_{\Gamma_{\rm{in}}}\rho^\gamma|u_B\cdot \nu(x)|\,d\sigma(x)dt
+\int_0^\tau\int_{\Gamma_{\rm{out}}}\frac{1}{\gamma-1}\rho^\gamma|u_B\cdot\nu(x)|\,d\sigma(x)dt\nonumber\\
\leq&\int_{\Omega}\Big(\frac{1}{2}\rho_0|u_0-u_{\infty}|^2+\frac{1}{\gamma-1}\rho_0^{\gamma}
+\int_{\mathbb{R}^3}\frac{|v|^2}{2}f_0\,dv\Big)\,dx\nonumber\\
&+\int^{\tau}_0\int_{\Omega}\big(-\rho^{\gamma}{\rm{div}}u_{\infty}-\mathbb{S}(\nabla u_{\infty}):\nabla(u-u_\infty)
-\rho u\cdot\nabla u_\infty\cdot(u-u_\infty)\big)\,dxdt\nonumber\\
&+\int^{\tau}_0\int_{\Gamma_{\rm{in}}}\frac{\gamma}{\gamma-1}\rho^{\gamma-1}\rho_{B}|u_B\cdot \nu(x)|\,d\sigma(x)dt
+3\int_0^\tau\int_{\Omega\times\mathbb{R}^3}f\,dxdvdt
-\int^{\tau}_0\int_{\Omega}(j-nu)\cdot u_\infty\,dxdt,
\end{align}
where $u_\infty(x)\in W^{1,\infty}(\Omega;\mathbb{R}^3)$ is an extension of $u_B(x)$, and satisfies
\begin{gather}
{\rm{div}}\,u_\infty\geq 0\;\;\textrm{a.e.}\;\;{\rm{in}}\;\;U^-_h\equiv\{x\in\Omega\big|{\rm{dist}}(x,\partial\Omega)<h\}\;\; {\rm{and}}\;\;h>0 \;\;{\rm{is\,\,small}}.\label{extension2}
\end{gather}
The existence of the extension $u_\infty$ of $u_B$ can be found in \cite{Gir}.
\end{defn}

We are now in a position to state our main result.
\begin{theorem}\label{main}
Let $\gamma>3/2$, $\Omega\subset\mathbb{R}^3$ be a three dimensional bounded domain with smooth boundary,
and $K:\Omega\times\Omega\rightarrow \mathbb{R}_+$ be symmetric, bounded,
$x\mapsto K(x,y)$ is Lipschitz continuous uniformly in $y$ and $K(x,y)\big|_{y\in \Sigma^+}=0$.
Suppose that the initial and boundary data satisfy
\begin{align}
&0\leq f_0\in L^1\cap L^\infty(\Omega\times\mathbb{R}^3),\;\;0\leq g\in L^1\cap L^\infty((0,T)\times\Sigma^-), \label{81}\\
&0<\int_{\Omega}\rho_0\,dx<\infty,\quad\int^T_0\int_{\Sigma^-}|v|^2g(t,x,v)|v\cdot \nu(x)|\,d\sigma(x)dvdt<\infty, \label{82}\\
&\int_{\Omega}\Big(\rho_0\frac{|u_0-u_\infty|^2}{2}+\frac{1}{\gamma-1}\rho_0^\gamma\Big)\,dx
+\int_{\Omega\times\mathbb{R}^3}\frac{|v|^2}{2}f_0\,dxdv<\infty,\label{83}
\end{align}
then for any $T>0$, the problem \eqref{CSNS}-\eqref{ub} admits at least one bounded energy weak solution $(f,\rho,u)$ on $[0,T]$.
\end{theorem}

\begin{remark}
 When $u_B\equiv 0$  in \eqref{ub}, we don't need the boundary condition  \eqref{rhob} for $\rho$. And
 our results still hold for this case and the proof {is} much simpler just by taking $u_\infty=u_B=0$ and
 $\rho_B=0$ in our arguments below.
\end{remark}

The existence of weak solutions is proved though a careful construction of approximation solutions and further compactness argument by using the energy estimate. The most technical part is to find an approximation solution sequence which satisfies the energy estimate. Compared to the analysis on the Vlasov-Fokker-Planck-Navier-Stokes equations in \cite{LL},
the main difficulties come from the nonlinear alignment force $fL[f]$ in $\eqref{CSNS}_1$.
This additional term makes it impossible to use the weak existence result for linear Vlasov-Fokker-Planck system obtained
in \cite{C} directly to construct approximate solutions.
To overcome these difficulties, we adopt some methods developed in \cite{KMT} and apply a special case of Leray-Schauder fixed
point theorem. On the one hand, because of the coupling with Navier-Stokes system, the approximate problem of the Cucker-Smale equation and the definition of fixed point
operator are much more complicated than those used in \cite{KMT}, where no fluid system is coupled.
On the other hand, due to the boundary effect, the compactness of the operator has to be proved with the help of Aubin-Lions lemma instead of Arzela-Ascoli Theorem. More details can be found in Section \ref{app}.
The additional nonlinear alignment force $fL[f]$ brings also further difficulties in the compactness argument, i.e. taking limit in regularized problem. The velocity-averaging lemma and the energy inequality make it possible to finish the argument.

The arrangement of this paper is as follows.
The approximate solutions to problem \eqref{CSNS}-\eqref{ub} are constructed in Section \ref{app}.
Furthermore, the existence of weak solutions to problem \eqref{CSNS}-\eqref{ub} is obtained
by means of uniform estimates and compactness argument in Section \ref{l}.

\section{Approximate solutions}\label{app}

We use many of the regularization terms already used in the literature \cite{CJN}, where only the compressible Navier-Stokes system has been studied, and obtain the following approximate system:
\begin{align}
&\partial_tf+v\cdot \nabla_xf+{\rm{div}}_v\big((u-v)f\big)+{\rm{div}}_v(fL[f])-\Delta_vf=0,\label{3} \\
&\partial_t\rho+{\rm{div}}(\rho u)=\varepsilon\Delta\rho,\label{5} \\
&\partial_t(\rho u)+{\rm{div}}(\rho u\otimes u)+\nabla\rho^{\gamma}+\delta\nabla\rho^\beta
+\varepsilon\nabla\rho\cdot\nabla u
\nonumber\\
&\qquad\quad ={\rm{div}}\mathbb{S}(\nabla u)
+\varepsilon{\rm{div}}(|\nabla(u-u_\infty)|^2\nabla(u-u_\infty))
-\int_{\mathbb{R}^3}(u-v)f\,dv,\label{6}
\end{align}
with the initial {condition}:
\begin{align}\label{7}
\big(f(0,x,v),\rho(0,x),u(0,x)\big)=(f_{0}(x,v),\rho_0(x),u_0(x)\big),
\end{align}
and the boundary conditions:
\begin{align}
&\gamma^-f(t,x,v)\big|_{(0,T)\times\Sigma^-}=g(t,x,v),\label{8}\\
&(-\varepsilon\nabla\rho+\rho u)\cdot \nu(x)\big|_{(0,T)\times\partial\Omega}=
\left\{
       \begin{array}{ll}
       \rho_Bu_B\cdot \nu(x) & \;\;\;{\rm{on}}\;\;\;\Gamma_{\rm{in}},\\
       \rho u_B\cdot \nu(x)&\;\;\;{\rm{on}}\;\;\;\partial\Omega\backslash\Gamma_{\rm{in}},
       \end{array}
\right.\label{9}\\
&u(t,x)\big|_{(0,T)\times\partial \Omega}=u_B(x),\label{10}
\end{align}
where $\varepsilon>0$, $\delta>0$ and $\beta>{\rm{max}}\{\gamma,9/2\}$,
and the initial and boundary data satisfy \eqref{81}-\eqref{83} and
\begin{equation}\label{11}
\left\{
\begin{aligned}
  &u_0\in L^2(\Omega),\quad \rho_0\in W^{1,2}(\Omega),\\
  &0<\underline{\rho}\leq\rho_0(x)\leq\bar{\rho}<\infty,\quad x\in\Omega,\\
  &0<\underline{\rho}\leq\rho_B(x)\leq\bar{\rho}<\infty,\quad x\in\Gamma_{\rm{in}}.
\end{aligned}
\right.
\end{equation}
We still use $(\rho_0,u_0,\rho_B)$ to denote the initial and boundary data to system \eqref{3}-\eqref{6} and will emphasize the dependence
of the parameter if it is needed.

Section 2 is intended to prove that the solutions of problem \eqref{3}-\eqref{10} exist. To achieve this, by using Galerkin's method we build up an approximation of this solution in subsection \ref{N} and then show that this sequence has an accumulation point which solves problem \eqref{3}-\eqref{10} in subsection \ref{lN}.

\subsection{Approximate solutions to problem \eqref{3}-\eqref{10}}\label{N}

To establish the existence of approximate solutions to problem \eqref{3}-\eqref{10} we employ the Galerkin's method.
We introduce a finite dimensional space $X={\rm{span}}\{\Psi_i\}_{i=1}^N$, where the smooth functions
$\Psi_i(x)\,(1\leq i\leq N)$ are orthonormal in $L^2(\Omega)$.
For a given triplet $(E,q,\tilde u)\in \{L^\infty((0,T)\times\Omega)\}^2\times L^2((0,T)\times\Omega)$, we consider
\begin{align}
&\partial_t f+v\cdot \nabla_xf+{\rm{div}}_v((E-qv)f)-\Delta_v f+{\rm{div}}_v\big((\tilde u\chi_{\{|\tilde u|\leq N \}}-v)f \big)=0,\label{r1}\\
&f|_{t=0}= f_{0,N},\quad \gamma^-f|_{(0,T)\times\Sigma^-}= g_N,\label{r2}\\
&\partial_t\rho+{\rm{div}}_x(\rho u)=\varepsilon\Delta_x\rho, \label{108}\\
&\rho|_{t=0}=\rho_0(x),\quad
(-\varepsilon\nabla\rho+\rho u)\cdot \nu(x)\big|_{(0,T)\times\partial\Omega}=
\left\{
       \begin{array}{llr}
       \rho_Bu_B\cdot \nu(x)& \;\;{\rm{on}}\;\;\Gamma_{\rm{in}},\\
       \rho u_B\cdot \nu(x)&\;\;{\rm{on}}\;\;\partial\Omega\backslash\Gamma_{\rm{in}},
       \end{array}
\right.\label{109}\\
&\partial_t(\rho u)+{\rm{div}}(\rho u\otimes u)+\nabla\rho^{\gamma}+\delta\nabla\rho^{\beta}+\varepsilon\nabla\rho\cdot\nabla u\nonumber\\
&\qquad ={\rm{div}}\mathbb{S}(\nabla u)+\varepsilon{\rm{div}}(|\nabla(u-u_\infty)|^{2}\nabla(u-u_\infty))
-\chi_{\{|\tilde u|\leq N \}}\int_{\mathbb{R}^3}(\tilde u-v)f\,dv, \label{110}\\
&u|_{t=0}=u_0(x),\quad u(t,x)\big|_{(0,T)\times\partial\Omega}=u_B(x),\label{111}
\end{align}
where $\chi:\mathbb{R}_+\rightarrow [0,1]$ is continuous indicator function. The initial and boundary conditions $f_{0,N}$ and $g_N$ are taken to be approximate sequences of $f_0$ and $g$, respectively. Furthermore, they satisfy \eqref{81}-\eqref{83} uniformly with respect to $N$ and
\begin{gather*}
\int_{\Omega\times\mathbb{R}^3}|v|^\kappa f_{0,N}(x,v)\,dxdv<\infty,\quad \forall\; \kappa\in[0,\kappa_0]\;\; {\rm{with}}\;\; \kappa_0\geq 5,\\
\int_0^T\int_{\Sigma^-}|v|^\kappa g_{N}|v\cdot\nu(x)|\,d\sigma(x)dvdt<\infty,\quad \forall\; \kappa\in[0,\kappa_0]\;\; {\rm{with}}\;\; \kappa_0\geq 5.
\end{gather*}
Based on the result obtained in \cite{LLS}, problem \eqref{r1}-\eqref{r2} possesses a unique global weak solution
$f\in C([0,T];L^1(\Omega\times\mathbb{R}^3))$ satisfying
\begin{align}
&\frac{d}{dt}\int_{\Omega\times\mathbb{R}^3}f\,dxdv=-\int_{\Sigma^\pm}(v\cdot\nu(x))\gamma^\pm f\,d\sigma(x)dv,\label{estimate1}\\
&\|f\|_{L^\infty((0,T)\times\Omega\times\mathbb{R}^3)}\leq
e^{\frac{C_1 T}{p'}}(\|f_{0,N}\|_{L^\infty(\Omega\times\mathbb{R}^3)}+\|g_N\|_{L^\infty((0,T)\times\Sigma^-)}),
\label{estimate2}\\
&\frac{d}{dt}\int_{\Omega\times\mathbb{R}^3}|v|^\kappa f\,dxdv+\int_{\Sigma^\pm}(v\cdot\nu(x))|v|^\kappa\gamma^\pm f\,d\sigma(x)dv
-\kappa(\kappa+1)\int_{\Omega\times\mathbb{R}^3}|v|^{\kappa-2}f\,dxdv\nonumber\\
&=\kappa\int_{\Omega\times\mathbb{R}^3}\big(f|v|^{\kappa-2}(E+\tilde u\chi_{\{|\tilde u|\leq N \}})\cdot v
-f|v|^\kappa (q+1)\big)\,dxdv,\quad 1\leq \kappa\leq \kappa_0,
\label{estimate3}
\end{align}
where $C_1$ is a positive constant dependent on $\|q\|_{L^\infty((0,T)\times\Omega)}$.
In order to deal with the fluid system, we need the following lemma.
\begin{lemma}\label{1'}
If one has
\begin{align*}
\|f\|_{L^\infty((0,T)\times\Omega\times\mathbb{R}^3)}&\leq M, \\
\int_{\Omega\times \mathbb{R}^3}|v|^\kappa f(t,x,v)\,dxdv&\leq M,\;\; t\in[0,T],\;\; \kappa\in[0,\kappa_0],
\end{align*}
then there exists a constant $C(M)$ such that
\begin{align*}
\|n(t)\|_{L^p(\Omega)}&\leq C(M), \quad p\in\Big[1,\frac{\kappa_0+3}{3}\Big],\\
\|j(t)\|_{L^p(\Omega)}&\leq C(M), \quad p\in\Big[1,\frac{\kappa_0+3}{4}\Big],
\end{align*}
for all $t\in[0,T]$.
\end{lemma}
\begin{proof}
Notice that
\begin{align*}
n=\int_{\{|v|>G\}}\frac{G^{\kappa_0}}{G^{\kappa_0}}f\,dv+\int_{\{|v|\leq G\}}f\,dv
\leq \frac{1}{G^{\kappa_0}}\int_{\mathbb{R}^3}|v|^{\kappa_0}f\,dv+\|f\|_{L^\infty(\mathbb{R}^3)}(2G)^3.
\end{align*}
Taking
\begin{align*}
G=\Big(\int_{\mathbb{R}^3}|v|^{\kappa_0}f\,dv\Big)^\frac{1}{\kappa_0+3},
\end{align*}
we obtain
\begin{align*}
\|n\|_{L^{\frac{\kappa_0+3}{3}}(\Omega)}^{\frac{\kappa_0+3}{3}}
&\leq \int_{\Omega} (1+8\|f\|_{L^\infty(\mathbb{R}^3)})^{\frac{\kappa_0+3}{3}}\Big(\int_{\mathbb{R}^3}|v|^{\kappa_0}f\,dv \Big)\,dx\\
&\leq \Big(1+8\|f\|_{L^\infty(\Omega\times\mathbb{R}^3)} \Big)^{\frac{\kappa_0+3}{3}}\int_{\Omega\times\mathbb{R}^3}|v|^{\kappa_0}f\,dvdx.
\end{align*}
An argument similar to the one used above shows that the second statement of Lemma \ref{1'} holds.
\end{proof}
Now we turn to Navier-Stokes system \eqref{108}-\eqref{111}
when $f$ is the weak solution to problem \eqref{r1}-\eqref{r2}.
Due to \eqref{estimate2}, \eqref{estimate3}, and Lemma \ref{1'}, we have
\begin{align*}
\Big\| \chi_{\{|\tilde u|\leq N \}}\int_{\mathbb{R}^3}(\tilde u-v)f\,dv\Big\|_{L^\infty(0,T;L^2(\Omega))}\leq C.
\end{align*}
Thanks to this observation, we perform the same reasoning as the that in \cite{CJN} to construct a unique approximate solution $(\rho_N,u_N)$
to problem \eqref{108}-\eqref{111}.
And $u_N\in C([0,T];X)$ can be written as
\begin{align*}
u_N(t,x)=\sum^N_{i=1}\zeta_i(t)e_i(x),
\end{align*}
where $\zeta_i(t)$ $(i=1,2,\cdot\cdot\cdot, N)$ are functions of $t$.
Furthermore, $(\rho_N,u_N)$ verifies, for any $(\tau,x)\in (0,T)\times\Omega$,
\begin{equation}\label{123}
\inf_{x\in\Omega}\rho_0(x)e^{-\int_0^T\|{\rm{div}}u_N\|_{L^\infty(\Omega)}\,dt}\leq \rho_N(\tau,x)\leq
\sup_{x\in \Omega}\rho_0(x) e^{\int_0^T\|{\rm{div}}u_N\|_{L^\infty(\Omega)}\,dt},
\end{equation}
and,
\begin{align}\label{122}
&\int_{\Omega}\Big(\frac{1}{2}\rho_{N}|u_{N}-u_\infty|^2+\frac{1}{\gamma-1}\rho_{N}^{\gamma}
+\frac{\delta}{\beta-1}\rho_{N}^{\beta}+\frac{1}{2}\rho_{N}^{2}\Big)\,dx
+\varepsilon\int_0^{\tau}\int_{\Omega}|\nabla\rho_{N}|^2\,dxdt\nonumber\\
&+\frac{1}{2}\int^\tau_0\int_{\partial\Omega}\rho_{N}^2|u_B\cdot \nu(x)|\,d\sigma(x)dt
+\varepsilon\int^{\tau}_0\int_{\Omega}(\gamma\rho_{N}^{\gamma-2}+\delta\beta\rho_{N}^{\beta-2})|\nabla\rho_{N}|^2\,dxdt\nonumber\\
&+\int_0^{\tau}\int_{\Omega}\mathbb{S}(\nabla(u_N-u_\infty)):\nabla(u_N-u_\infty)\,dxdt
+\int_0^\tau\int_{\Gamma_{\rm{out}}}\Big(\frac{1}{\gamma-1}\rho_{N}^\gamma+\frac{\delta}{\beta-1}\rho_{N}^\beta \Big)|u_B\cdot \nu(x)|\,d\sigma(x)dt
\nonumber\\
&+\varepsilon\int_0^\tau\int_{\Omega}|\nabla(u_{N}-u_\infty)|^4\,dxdt
+\int_0^\tau\int_{\Gamma_{\rm{in}}}(\rho_{N}^\gamma+\delta\rho_{N}^\beta)|u_B\cdot\nu(x)|\,d\sigma(x)dt
\nonumber\\
\leq\;&\int_{\Omega}\Big(\frac{1}{2}\rho_0|u_0-u_\infty|^2+\frac{1}{\gamma-1}\rho^{\gamma}_0
+\frac{\delta}{\beta-1}\rho_0^{\beta}+\frac{1}{2}\rho_0^2\Big)\,dx
+\int^\tau_0\int_{\Gamma_{\rm{in}}}\rho_{N}\rho_B|u_B\cdot \nu(x)|\,d\sigma(x)dt \nonumber\\
&+\int^{\tau}_0\int_{\Gamma_{\rm{in}}}\Big(\frac{\gamma}{\gamma-1}\rho_{N}^{\gamma-1}
+\frac{\delta\beta}{\beta-1}\rho_{N}^{\beta-1}\Big)\rho_B|u_B\cdot \nu(x)|\,d\sigma(x)dt
-\frac{1}{2}\int^{\tau}_0\int_{\Omega}\rho_{N}^2{\rm{div}}u_{N}\,dxdt\nonumber\\
&+\int^{\tau}_0\int_{\Omega}\big(-(\rho_{N}^{\gamma}+\delta\rho_{N}^{\beta}){\rm{div}}u_\infty
-\mathbb{S}(\nabla u_\infty):\nabla(u_{N}-u_\infty)
-\rho_{N} u_{N}\cdot\nabla u_\infty\cdot(u_{N}-u_\infty)\nonumber\\
&\qquad+\varepsilon\nabla\rho_{N}\cdot\nabla(u_{N}-u_\infty)\cdot u_\infty  \big)\,dxdt
+\int^{\tau}_0\int_{\Omega\times\mathbb{R}^3}f(v-\tilde{u})\chi_{\{|\tilde u|\leq N \}}\cdot (u_{N}-u_\infty)\,dxdvdt.
\end{align}
This result {enables} us to define an operator $\mathscr{T}$:
\begin{align*}
\mathscr{T}:[L^\infty((0,T)\times\Omega)]^2\times L^2((0,T)\times\Omega)&\rightarrow [L^\infty((0,T)\times\Omega)]^2\times L^2((0,T)\times\Omega),\\
(E,q,\tilde u)&\mapsto (\mathscr{T}_1,\mathscr{T}_2,\mathscr{T}_3),
\end{align*}
where
\begin{align*}
&\mathscr{T}_1(E,q,\tilde u)=\int_{\Omega\times\mathbb{R}^3}K(x,y)f(t,y,w)w\,dwdy,\\
&\mathscr{T}_2(E,q,\tilde u)=\int_{\Omega\times\mathbb{R}^3}K(x,y)f(t,y,w)\,dwdy,\\
&\mathscr{T}_3(E,q,\tilde u)=u_N.
\end{align*}
Therefore, the existence of the approximate solution to problem \eqref{3}-\eqref{10} is reduced to prove the existence of a fixed point for the operator $\mathscr{T}$.
We will use the following special case of Leray-Schauder fixed point theorem {\cite{GT}}.
\begin{lemma}\label{LS}
Let $\mathscr{T}$ be a continuous and compact mapping of a Banach space $\mathscr{B}$ into itself. Suppose there exists a constant $M$ such that
\begin{align*}
\|w\|_{\mathscr{B}}<M,
\end{align*}
for all $w\in\mathscr{B}$ and $\alpha\in [0,1]$ satisfying $w=\alpha\mathscr{T}w$, then $\mathscr{T}$ has a fixed point.
\end{lemma}
Let $\{(E_i,q_i,\tilde u_i) \}$ be uniformly bounded in $[L^\infty((0,T)\times\Omega)]^2\times L^2((0,T)\times\Omega)$
and $\{(f_i,\rho_N^i,u_N^i)\}$ be the corresponding sequence of solutions to system \eqref{r1}-\eqref{111}.
Following similar argument to that has been used in subsection 2.3 of \cite{LL}, the estimates in \eqref{123}-\eqref{122} imply that
\begin{align}
&\|u_N^i\|_{L^2(0,T;H^1(\Omega))}\leq C_2,\label{124}\\
&\|\partial_t u_N^i\|_{L^p(0,T;W^{-1,p}(\Omega))}\leq C_2\quad {\rm{for \;\;some\;\;}} p\in(1,\infty),\label{125}
\end{align}
where $C_2$ depends on $N$, but is independent of $i$.
The inequalities \eqref{124} and \eqref{125} allow us to use Aubin-Lions lemma to get
\begin{align*}
u_N^i\rightarrow u_N \quad {\rm{in}}\quad L^2((0,T)\times\Omega) \quad {\rm{as}}\quad i\rightarrow+\infty.
\end{align*}

From \eqref{estimate1}-\eqref{estimate3}, we have
\begin{align}\label{126}
&\Big\|\nabla_x\int_{\Omega\times\mathbb{R}^3}K(x,y)f_i(t,y,w)w\,dwdy \Big\|_{L^\infty((0,T)\times\Omega)}\nonumber\\
\leq &\sup_{x,y}|\nabla_x K(x,y)|\sup_t\Big(\int_{\Omega\times\mathbb{R}^3}|w|^2f_i\,dwdy \Big)^\frac{1}{2}
\Big(\int_{\Omega\times\mathbb{R}^3}f_i\,dwdy \Big)^\frac{1}{2}\nonumber\\
\leq&\, C_2.
\end{align}
By means of \eqref{r1} and the assumption $K(x,y)\big|_{y\in \Sigma^+}=0$,
we obtain that
\begin{align}\label{127}
&\Big\|\partial_t\int_{\Omega\times\mathbb{R}^3}K(x,y)f_i(t,y,w)w\,dwdy \Big\|_{L^\infty((0,T)\times\Omega)}\nonumber\\
 \leq&\,  \Big\|\int_{\Omega\times\mathbb{R}^3}K(x,y)|w|^2\nabla_y f_i(t,y,w)\,dwdy  \Big\|_{L^\infty((0,T)\times\Omega)}\nonumber\\
&+ C_2\,\Big\|\int_{\Omega\times\mathbb{R}^3}K(x,y)(E_if_i-q_iwf_i)\,dwdy \Big\|_{L^\infty((0,T)\times\Omega)}\nonumber\\
&+C_2\,\Big\|\int_{\Omega\times\mathbb{R}^3}\big(\tilde u_i\chi_{\{|\tilde u_i|\leq N \}}f_i-wf_i \big)\,dwdy \Big\|_{L^\infty((0,T)\times\Omega)}\nonumber\\
\leq&\, C_2.
\end{align}
The inequalities \eqref{126}-\eqref{127} and the fact that $W^{1,\infty}(\Omega)\hookrightarrow\hookrightarrow L^\infty(\Omega)$
allow us to apply Aubin-Lions lemma to
infer that
\begin{align*}
\int_{\Omega\times\mathbb{R}^3}K(x,y)f_i(t,y,w)w\,dwdy\rightarrow
\int_{\Omega\times\mathbb{R}^3}K(x,y)f(t,y,w)w\,dwdy\quad {\rm{in}}\quad  L^\infty((0,T)\times\Omega).
\end{align*}
Actually, Aubin-Lions lemma makes it possible to yield the strong convergence in $L^\infty((0,T)\times\Omega)$ of the left-hand side of the above relation towards some limit $\xi$. Separately, the estimates in \eqref{estimate2}, \eqref{estimate3} and boundedness of $w f_i$ in $L^5((0,T)\times\Omega)$ imply that (up to subsequences) the above convergence is also weakly in the same space. Therefore, one can identify that $\xi$ is exactly $\int_{\Omega\times\mathbb{R}^3}K(x,y)f(t,y,w)w\,dwdy$.

Similarly, we can derive that
\begin{align*}
\int_{\Omega\times\mathbb{R}^3}K(x,y)f_i(t,y,w)\,dwdy\rightarrow
\int_{\Omega\times\mathbb{R}^3}K(x,y)f(t,y,w)\,dwdy\quad {\rm{in}}\quad  L^\infty((0,T)\times\Omega).
\end{align*}
Therefore, we have proved that the operator $\mathscr{T}$ is compact.

For any $(E,q,\tilde u)\in \{L^\infty((0,T)\times\Omega)\}^2\times L^2((0,T)\times\Omega)$ satisfying
$(E,q,\tilde u)=\alpha\mathscr{T}(E,q,\tilde u)$, $\alpha\in [0,1]$,
let $f$ be the weak solution to the following equation
\begin{align*}
\partial_t f+v\cdot \nabla_x f+ {\rm{div}}_v\big( (\alpha u_N \chi_{\{\alpha |u_N|\leq N \}}-v)f\big)
-\Delta_v f+\alpha {\rm{div}}_v(fL[f])=0.
\end{align*}
It is easy to see that
\begin{align*}
&\frac{d}{dt}\int_{\Omega\times\mathbb{R}^3}\frac{|v|^2}{2}f\,dxdv+\int_{\Omega\times\mathbb{R}^3}|v|^2f\,dxdv
+\int_{\Omega\times\mathbb{R}^3}|v-\alpha u_N\chi_{\{\alpha|u_N|\leq N \} }|^2f\,dxdv   \\
&+\frac{\alpha}{2}\int_{\Omega\times\mathbb{R}^3}\int_{\Omega\times\mathbb{R}^3}K(x,y)f(t,x,v)f(t,y,w)|w-v|^2\,dydwdxdv\\
=&3\int_{\Omega\times\mathbb{R}^3}f\,dxdv
-\frac{1}{2}\int_{\Sigma^\pm}(v\cdot\nu(x))|v|^2\gamma^\pm f\,d\sigma(x)dv\\
&-\int_{\Omega\times\mathbb{R}^3}f(v-\alpha u_N\chi_{\{\alpha|u_N|\leq N \}})\cdot \alpha u_N\chi_{\{\alpha|u_N|\leq N \}}\,dxdv.
\end{align*}
Owing to the above equality and \eqref{estimate1}, we have
\begin{align*}
\left\||v|^2f \right\|_{L^\infty(0,T;L^1(\Omega\times\mathbb{R}^3))}\leq C_3,
\end{align*}
where $C_3$ depends on $N$, but is independent of $\alpha$.
Together with \eqref{estimate1}, we conclude that
\begin{align*}
\|E\|_{L^\infty((0,T)\times\Omega)}&\leq \|\mathscr{T}_1(E,q,\tilde u)\|_{L^\infty((0,T)\times\Omega)}\\
&\leq\sup_{x,y}|K(x,y)|\sup_t\Big(\int_{\Omega\times\mathbb{R}^3}|w|^2f\,dwdy\Big)^\frac{1}{2} \Big(\int_{\Omega\times\mathbb{R}^3}f\,dwdy \Big)^\frac{1}{2}\\
&\leq C_3,\\
\|q\|_{L^\infty((0,T)\times\Omega)}
&\leq \sup_{x,y}|K(x,y)|\sup_t\int_{\Omega\times\mathbb{R}^3}f\,dwdy\leq C_3.
\end{align*}
For the fluid part, the last term on the right-hand side of \eqref{122} with $\tilde u=\alpha u_N$ can be estimated in the following way:
\begin{align*}
&\int_0^\tau\int_{\Omega\times\mathbb{R}^3}f(v-\alpha u_N)\chi_{\{\alpha|u_N|\leq N \}}\cdot(u_N-u_\infty)\,dxdvdt\\
\leq &\frac{\varepsilon}{2}\int_0^\tau\int_{\Omega\times\mathbb{R}^3}|\nabla(u_N-u_\infty)|^4\,dxdt
+C_3\Big(\int_0^\tau\int_\Omega |j|^\frac{5}{4}\,dxdt+\int_0^\tau\int_\Omega n^\frac{5}{4}\,dxdt \Big).
\end{align*}
Plugging the above inequality into \eqref{122}, we infer that
\begin{align*}
\|\tilde u\|_{L^2((0,T)\times\Omega)}\leq \|u_N\|_{L^2((0,T)\times\Omega)}\leq C_3.
\end{align*}

Therefore, we can apply Lemma \ref{LS} to yield that there exists a fixed point $(E,q,u_N)$ such that $(E,q,u_N)=\mathscr{T}(E,q,u_N)$.
And the estimates \eqref{estimate1}-\eqref{estimate3} {hold} with
{$$
	q=\int_{\Omega\times\mathbb{R}^3}K(x,y)f_N(t,y,w)\,dwdy,\qquad
	E=\int_{\Omega\times\mathbb{R}^3}K(x,y)f_N(t,y,w)w\,dwdy.
	$$
}
Thus, we arrive at the conclusion that the triplet $(f_N,\rho_N,u_N)$ satisfies\\
1. Weak formulation of the Cucker-Smale equation: for any $\varphi\in C^\infty_c([0,T)\times\bar\Omega\times\mathbb{R}^3)$ such that $\varphi=0$ on $(0,T)\times\Sigma^+$, it holds
\begin{align}\label{119}
&\int^{T}_0\int_{\Omega\times\mathbb{R}^3}f_N\big(\partial_t\varphi+v\cdot\nabla_x\varphi
+(u_{N} \chi_{\{|u_N|\leq N \}}-v)\cdot\nabla_v\varphi
+L[f_N]\cdot \nabla_v\varphi
+\Delta_v\varphi\big)\,dxdvdt\nonumber\\
&=-\int_{\Omega\times\mathbb{R}^3}f_{0,N}\varphi(0,x,v)\,dxdv
+\int_0^T\int_{\Sigma^-}(v\cdot \nu(x))g_N\varphi\,d\sigma(x)dvdt;
\end{align}
2. Weak formulation of the continuity equation: for any $\psi\in C^\infty_c([0,T)\times(\Omega\cup\Gamma_{\rm{in}}))$, it holds
\begin{align}\label{120}
&\int^{T}_0\int_{\Omega}(\rho_{N}\partial_t\psi+\rho_{N} u_{N}\cdot\nabla\psi-\varepsilon\nabla\rho_{N}\cdot \nabla\psi)\,dxdt
+\int_{\Omega}\rho_0\psi(0,x)\,dx\nonumber\\
&\qquad\qquad\qquad=\int_0^T\int_{\Gamma_{\rm{in}}}\rho_B u_B\cdot \nu(x)\psi\,d\sigma(x)dt;
\end{align}
3. Weak formulation of the momentum balance equation: for any $\phi\in C^\infty_c((0,T)\times\Omega;\mathbb{R}^3)$, it holds
\begin{align}\label{121}
\int^{T}_{0}&\int_{\Omega}\big(\rho_{N} u_{N}\cdot\partial_t\phi+(\rho_{N} u_{N}\otimes u_{N}):\nabla\phi+\rho_{N}^\gamma{\rm{div}}\phi
+\delta\rho_{N}^\beta{\rm{div}}\phi
-\varepsilon\nabla\rho_{N}\cdot\nabla u_{N}\phi\nonumber\\
&-\varepsilon|\nabla(u_{N}-u_\infty)|^2\nabla(u_{N}-u_\infty):\nabla\phi
-\mathbb{S}(\nabla u_{N}):\nabla\phi
+(j_N-n_N u_{N})\chi_{\{|u_N|\leq N \}}\cdot\phi\big)\,dxdt\nonumber\\
&+\int_{\Omega}\rho_0u_0\cdot \phi(0,x)\,dx=0.
\end{align}
Furthermore, $(f_N,\rho_N,u_N)$ verifies
\begin{align}
&\partial_t\int_{\mathbb{R}^3}f_N\,dv+{\rm{div}}_x\int_{\mathbb{R}^3}vf_N\,dv=0,\label{112}\\
&\frac{d}{dt}\int_{\Omega\times\mathbb{R}^3}f_N\,dxdv=-\int_{\Sigma^\pm}(v\cdot\nu(x))\gamma^\pm f_N\,d\sigma(x)dv,\label{113}\\
&\|f_N\|_{L^\infty((0,T)\times\Omega\times\mathbb{R}^3)}
\leq e^{\frac{CT}{p'}}(\|f_{0,N}\|_{L^\infty(\Omega\times\mathbb{R}^3)}+\|g_N\|_{L^\infty((0,T)\times\Sigma^-)}),\label{114}
\end{align}
where $C$ depends on $\|K\|_{L^\infty(\Omega\times\Omega)}$,
and for any $\tau \in(0,T)$, it holds
\begin{align}\label{e1}
&\int_{\Omega}\Big(\frac{1}{2}\rho_{N}|u_{N}-u_\infty|^2+\frac{1}{\gamma-1}\rho_{N}^{\gamma}
+\frac{\delta}{\beta-1}\rho_{N}^{\beta}+\frac{1}{2}\rho_{N}^{2}
+\int_{\mathbb{R}^3}\frac{|v|^2}{2}f_N\,dv\Big)\,dx\nonumber\\
&+\varepsilon\int_0^{\tau}\int_{\Omega}|\nabla\rho_{N}|^2\,dxdt
+\frac{1}{2}\int^\tau_0\int_{\partial\Omega}\rho_{N}^2|u_B\cdot \nu(x)|\,d\sigma(x)dt\nonumber\\
&+\int_0^{\tau}\int_{\Omega}\mathbb{S}(\nabla(u_N-u_\infty)):\nabla(u_N-u_\infty)\,dxdt
+\int_0^\tau\int_{\Gamma_{\rm{out}}}\Big(\frac{1}{\gamma-1}\rho_{N}^\gamma+\frac{\delta}{\beta-1}\rho_{N}^\beta\Big)|u_B\cdot\nu(x)|\,d\sigma(x)dt
\nonumber\\
&+\varepsilon\int^{\tau}_0\int_{\Omega}(\gamma\rho_{N}^{\gamma-2}+\delta\beta\rho_{N}^{\beta-2})|\nabla\rho_{N}|^2\,dxdt
+\int^{\tau}_0\int_{\Sigma^-}(v\cdot \nu(x))\frac{|v|^2}{2}g_N\,d\sigma(x)dvdt\nonumber\\
&+\frac{1}{2}\int_0^\tau\int_{\Omega\times\mathbb{R}^3}\int_{\Omega\times\mathbb{R}^3}K(x,y)f_N(t,x,v)f_N(t,y,w)|w-v|^2\,dydwdxdvdt\nonumber\\
&+\varepsilon\int_0^{\tau}\int_{\Omega}|\nabla(u_{N}-u_\infty)|^4\,dxdt
+\int_0^\tau\int_{\Gamma_{\rm{in}}}(\rho_{N}^\gamma+\delta\rho_{N}^\beta)|u_B\cdot\nu(x)|\,d\sigma(x)dt\nonumber\\
\leq\;&\int_{\Omega}\Big(\frac{1}{2}\rho_0|u_0-u_\infty|^2+\frac{1}{\gamma-1}\rho^{\gamma}_0
+\frac{\delta}{\beta-1}\rho_0^{\beta}+\frac{1}{2}\rho_0^2
+\int_{\mathbb{R}^3}\frac{|v|^2}{2}f_{0,N}\,dv\Big)\,dx \nonumber\\
&+\int^{\tau}_0\int_{\Gamma_{\rm{in}}}\Big(\frac{\gamma}{\gamma-1}\rho_{N}^{\gamma-1}
+\frac{\delta\beta}{\beta-1}\rho_{N}^{\beta-1}\Big)\rho_B|u_B\cdot \nu(x)|\,d\sigma(x)dt\,d\sigma(x)dt
-\frac{1}{2}\int^{\tau}_0\int_{\Omega}\rho_{N}^2{\rm{div}}u_{N}\,dxdt\nonumber\\
&+\int^{\tau}_0\int_{\Omega}\big(-(\rho_{N}^{\gamma}+\delta\rho_{N}^{\beta}){\rm{div}}u_\infty
-\mathbb{S}(\nabla u_\infty):\nabla(u_{N}-u_\infty)
-\rho_{N} u_{N}\cdot\nabla u_\infty\cdot(u_{N}-u_\infty)\nonumber\\
&\qquad +\varepsilon\nabla\rho_{N}\cdot\nabla(u_{N}-u_\infty)\cdot u_\infty  \big)\,dxdt
-\int^{\tau}_0\int_{\Omega}(j_N-n_Nu_{N})\chi_{\{|u_N|\leq N \}}\cdot u_\infty\,dxdt\nonumber\\
&+3\int^{\tau}_0\int_{\Omega\times\mathbb{R}^3}f_N\,dxdvdt
+\int_0^\tau\int_{\Gamma_{\rm{in}}}\rho_N\rho_B|u_B\cdot \nu(x)|\,d\sigma(x)dt.
\end{align}

\subsection{Compactness argument for $N\rightarrow +\infty$}\label{lN}

In this subsection, we perform the limit $N\rightarrow +\infty$ in the equalities \eqref{119}-\eqref{121}
with the help of \eqref{112}-\eqref{e1}. The proof is similar to that given in \cite{MV}, except the weak compactness of $f_N L[f_N]$. We will only focus on this term in the following discussion.

We start with the following velocity-averaging lemma, which is proved in \cite{KMT}.
\begin{lemma}\label{CS5}
Let $T>0$, $\{H_N\}$ and $\{G_N\}$ be bounded in $L^p_{loc}((0,T)\times\mathbb{R}^3\times\mathbb{R}^3)$ with $1<p<\infty$.
Suppose that $H_N$ and $G_N$ verify
\begin{align*}
\partial_t H_N+v\cdot \nabla_xH_N={\rm{div}}_v^\alpha G_N,\qquad
H_N|_{t=0}=H_0\in L^p(\mathbb{R}^3\times\mathbb{R}^3),
\end{align*}
for a multiindex $\alpha$, and
\begin{align*}
\sup_N\|H_N\|_{L^\infty((0,T)\times\mathbb{R}^3\times\mathbb{R}^3)}
+\sup_N\|(|x^2|+|v^2|)H_N\|_{L^\infty(0,T;L^1(\mathbb{R}^3\times\mathbb{R}^3))}<\infty,
\end{align*}
then, for any $\varphi(v)$ such that $|\varphi(v)|\leq C|v|$, the sequence $\{\int_{\mathbb{R}^3}H_N\varphi\,dv\}$ is relatively compact in
$L^q((0,T)\times\mathbb{R}^3)$ for any $1\leq q<\frac{5}{4}$.
\end{lemma}

The estimates in \eqref{e1} and Lemma \ref{1'} imply that
\begin{align}
&\left\||v|^2f_N\right\|_{L^\infty(0,T;L^1(\Omega\times\mathbb{R}^3))}\leq C,\label{116}\\
&\|n_N\|_{L^\infty(0,T;L^p(\Omega))}\leq C, \quad p\in\Big[1,\frac{5}{3}\Big],\label{117}\\
&\|j_N\|_{L^\infty(0,T;L^p(\Omega))}\leq C, \quad p\in\Big[1,\frac{5}{4}\Big].\label{118}
\end{align}
Applying Lemma \ref{CS5} with $G_N=\big(\nabla_v f_N-(u_N\chi_{\{|u_N|\leq N \}}-v+L[f_N])f_N\big)\chi_{\{\Omega\times\mathbb{R}^3\}}$
and $H_N=f_N\chi_{\{\Omega\times\mathbb{R}^3\}}$,
we deduce that there exists a subsequence $\{f_N\}$ (not relabeled) such that
\begin{align*}
\int_{\mathbb{R}^3}vf_N\,dv\rightarrow \int_{\mathbb{R}^3}vf\,dv \quad {\rm{in}}\quad L^p((0,T)\times\Omega),\quad p\in \Big[1,\frac{5}{4}\Big).
\end{align*}
Noticing the fact that
\begin{align*}
\int_{\Omega\times\mathbb{R}^3}K(x,y)f_N(t,y,w)\,dwdy\rightharpoonup\int_{\Omega\times\mathbb{R}^3}K(x,y)f(t,y,w)\,dwdy
\quad {\rm{in}}\quad L^p((0,T)\times\Omega),\quad p\in(1,\infty),
\end{align*}
we arrive at the following convergence
\begin{align*}
\int_0^T\int_{\Omega\times\mathbb{R}^3}&vf_N(t,x,v)\Big(\int_{\Omega\times\mathbb{R}^3}K(x,y)f_N(t,y,w)\,dwdy \Big)\cdot \nabla_v\varphi\,dxdvdt\\
&\rightarrow
\int_0^T\int_{\Omega\times\mathbb{R}^3}vf(t,x,v)\Big(\int_{\Omega\times\mathbb{R}^3}K(x,y)f(t,y,w)\,dwdy \Big)\cdot \nabla_v\varphi\,dxdvdt.
\end{align*}
Thanks to the weak lower semicontinuity of convex functions, we are able to take the limit $N\rightarrow+\infty$ in the energy inequality \eqref{e1} and obtain
\begin{prop}\label{p1}
Problem \eqref{3}-\eqref{10} possesses a global weak solution $(f,\rho,u)$ satisfying, for any $\tau\in(0,T)$,
\begin{align}\label{e2}
&\int_{\Omega}\Big(\frac{1}{2}\rho|u-u_\infty|^2+\frac{1}{\gamma-1}\rho^{\gamma}
+\frac{\delta}{\beta-1}\rho^{\beta}+\frac{1}{2}\rho^{2}
+\int_{\mathbb{R}^3}\frac{|v|^2}{2}f\,dv\Big)\,dx\nonumber\\
&+\varepsilon\int_0^{\tau}\int_{\Omega}|\nabla\rho|^2\,dxdt
+\varepsilon\int^{\tau}_0\int_{\Omega}(\gamma\rho^{\gamma-2}+\delta\beta\rho^{\beta-2})|\nabla\rho|^2\,dxdt
+\frac{1}{2}\int^\tau_0\int_{\partial\Omega}\rho^2|u_B\cdot \nu(x)|\,d\sigma(x)dt  \nonumber\\
&+\frac{1}{2}\int_0^\tau\int_{\Omega\times\mathbb{R}^3}\int_{\Omega\times\mathbb{R}^3}K(x,y)f(t,x,v)f(t,y,w)|w-v|^2\,dydwdxdvdt\nonumber\\
&+\int_0^{\tau}\int_{\Omega}\mathbb{S}(\nabla(u-u_\infty)):\nabla(u-u_\infty)\,dxdt
+\varepsilon\int_0^{\tau}\int_{\Omega}|\nabla(u-u_\infty)|^4\,dxdt
\nonumber\\
&+\int_0^\tau\int_{\Gamma_{\rm{in}}}(\rho^\gamma+\delta\rho^\beta)|u_B\cdot\nu(x)|\,d\sigma(x)dt
+\int_0^\tau\int_{\Gamma_{\rm{out}}}\Big(\frac{1}{\gamma-1}\rho^\gamma+\frac{\delta}{\beta-1}\rho^\beta \Big)|u_B\cdot\nu(x)|\,d\sigma(x)dt\nonumber\\
&+\int^{\tau}_0\int_{\Sigma^-}(v\cdot \nu(x))\frac{|v|^2}{2}g\,d\sigma(x)dvdt
-\int^\tau_0\int_{\Gamma_{\rm{in}}}\rho\rho_B|u_B\cdot \nu(x)|\,d\sigma(x)dt\nonumber\\
\leq\;&\int_{\Omega}\Big(\frac{1}{2}\rho_0|u_0-u_\infty|^2+\frac{1}{\gamma-1}\rho^{\gamma}_0
+\frac{\delta}{\beta-1}\rho_0^{\beta}+\frac{1}{2}\rho_0^2
+\int_{\mathbb{R}^3}\frac{|v|^2}{2}f_{0}\,dv\Big)\,dx \nonumber\\
&-\frac{1}{2}\int^{\tau}_0\int_{\Omega}\rho^2{\rm{div}}u\,dxdt
-\int^{\tau}_0\int_{\Omega}(j-nu)\cdot u_\infty\,dxdt
+3\int_{0}^\tau\int_{\Omega\times\mathbb{R}^3}f\,dxdvdt\nonumber\\
&+\int^{\tau}_0\int_{\Omega}\big(-(\rho^{\gamma}+\delta\rho^{\beta}){\rm{div}}u_\infty
-\mathbb{S}(\nabla u_\infty):\nabla(u-u_\infty)
-\rho u\cdot\nabla u_\infty\cdot(u-u_\infty)\nonumber\\
&\qquad +\varepsilon\nabla\rho\cdot\nabla(u-u_\infty)\cdot u_\infty  \big)\,dxdt
+\int_0^\tau\int_{\Gamma_{\rm{in}}}\Big(\frac{\gamma}{\gamma-1}\rho^{\gamma-1}
+\frac{\delta\beta}{\beta-1}\rho^{\beta-1} \Big)\rho_B|u_B\cdot\nu(x)|\,d\sigma(x)dt.
\end{align}
\end{prop}
\begin{remark}
The weak formulation of momentum equation \eqref{6} can be written as, for any $\phi\in C_c^\infty((0,T)\times\Omega;\mathbb{R}^3)$,
\begin{align*}
\int^{T}_{0}&\int_{\Omega}\big(\rho u\cdot\partial_t\phi+(\rho u\otimes u):\nabla\phi+\rho^\gamma{\rm{div}}\phi
+\delta\rho^\beta{\rm{div}}\phi
-\varepsilon\nabla\rho\cdot\nabla u\cdot\phi\nonumber\\
&-Z_\varepsilon:\nabla\phi
-\mathbb{S}(\nabla u):\nabla\phi
+(j-nu)\cdot\phi\big)\,dxdt
+\int_{\Omega}\rho_0u_0\cdot \phi(0,x)\,dx=0,
\end{align*}
where $Z_\varepsilon$ is the weak limit of the term $\varepsilon|\nabla(u_{N}-u_\infty)|^2\nabla(u_{N}-u_\infty)$ in the equality
\eqref{121}, that is,
\begin{align*}
\varepsilon|\nabla(u_{N}-u_\infty)|^2\nabla(u_{N}-u_\infty)\rightharpoonup Z_\varepsilon \quad
{\rm{in}} \quad L^\frac{4}{3}((0,T)\times\Omega),
\end{align*}
and
\begin{align*}
\|Z_\varepsilon\|_{L^\frac{4}{3}((0,T)\times\Omega)}\rightarrow 0 \quad {\rm{as}}\quad \varepsilon\rightarrow 0,
\end{align*}
where we have used the energy inequality \eqref{e1}.
\end{remark}

\section{Taking the limits $\varepsilon$, $\delta\to 0$.}\label{l}
Our goal of this section is to take the limits $\varepsilon\rightarrow 0$ and $\delta\rightarrow 0$ in the weak formulation of problem \eqref{3}-\eqref{10} and thus to complete the proof of Theorem \ref{main}.

\subsection{Taking the limit $\varepsilon\rightarrow 0$}\label{le}
The triplet $(f_\varepsilon,\rho_\varepsilon,u_\varepsilon)$ denotes the solutions constructed in Section \ref{app}.
A standard argument, for example in \cite{CJN}, based on the effective viscous flux identity, oscillations defect measure, and renormalization techniques for the continuity equation
makes it possible to take the limit $\varepsilon\rightarrow 0$ in the weak formulation of Navier-Stokes system
\eqref{5}-\eqref{10}.
Using the same argument as the one used in subsection \ref{lN}, we can easily pass to the limit $\varepsilon\rightarrow 0$ in the weak formulation of Cucker-Smale problem \eqref{3} and \eqref{7}-\eqref{10}.
In conclusion, we have proved the following proposition:
\begin{prop}
There exists a global weak solution $(f,\rho,u)$ to the following system:
\begin{align}
&\partial_t f+v\cdot \nabla_x f
+{\rm{div}}_v((u-v)f)+{\rm{div}}_v(fL[f])-\Delta_vf=0, \label{54}\\
&\partial_t\rho+{\rm{div}}_x(\rho u)=0, \label{56}\\
&\partial_t(\rho u)+{\rm{div}}(\rho u\otimes u)+\nabla\rho^{\gamma}+\delta\nabla\rho^{\beta}-{\rm{div}}\mathbb{S}(\nabla u)=
-\int_{\mathbb{R}^3}(u-v)f\,dv, \label{57}
\end{align}
subject to
\begin{align*}
&f(0,x,v)=f_{0}(x,v),\quad \rho(0,x)=\rho_0(x),\quad u(0,x)=u_0(x),\\
&\gamma^-f(t,x,v)\big|_{(0,T)\times\Sigma^-}=g(t,x,v),\\
&\rho(t,x)\big|_{(0,T)\times\Gamma_{\rm{in}}}=\rho_B(x),\quad
u(t,x)\big|_{(0,T)\times\partial\Omega}=u_B(x),
\end{align*}
where $(f_0,\rho_0,u_0)$ and $(g,\rho_B,u_B)$ satisfy \eqref{81}-\eqref{83} and \eqref{11}.
Furthermore, $(f,\rho,u)$ verifies
\begin{align*}
&\int_{\Omega}\Big(\frac{1}{2}\rho|u-u_\infty|^2+\frac{1}{\gamma-1}\rho^{\gamma}
+\frac{\delta}{\beta-1}\rho^{\beta}
+\int_{\mathbb{R}^3}\frac{|v|^2}{2}f\,dv\Big)\,dx\nonumber\\
&+\frac{1}{2}\int_0^\tau\int_{\Omega\times\mathbb{R}^3}\int_{\Omega\times\mathbb{R}^3}K(x,y)f(t,x,v)f(t,y,w)|w-v|^2\,dydwdxdvdt\nonumber\\
&+\int_0^{\tau}\int_{\Omega}\mathbb{S}(\nabla(u-u_\infty)):\nabla(u-u_\infty)\,dxdt
+\int^{\tau}_0\int_{\Sigma^-}(v\cdot \nu(x))\frac{|v|^2}{2}g\,d\sigma(x)dvdt\nonumber\\
&+\int_0^\tau\int_{\Gamma_{\rm{in}}}(\rho^\gamma+\delta\rho^\beta)|u_B\cdot\nu(x)|\,d\sigma(x)dt
+\int_0^\tau\int_{\Gamma_{\rm{out}}}\Big(\frac{1}{\gamma-1}\rho^\gamma+\frac{\delta}{\beta-1}\rho^\beta \Big)|u_B\cdot\nu(x)|\,d\sigma(x)dt\nonumber\\
\leq\;&\int_{\Omega}\Big(\frac{1}{2}\rho_0|u_0-u_\infty|^2+\frac{1}{\gamma-1}\rho^{\gamma}_0
+\frac{\delta}{\beta-1}\rho_0^{\beta}
+\int_{\mathbb{R}^3}\frac{|v|^2}{2}f_{0}\,dv\Big)\,dx \nonumber\\
&+\int^{\tau}_0\int_{\Gamma_{\rm{in}}}\Big(\frac{\gamma}{\gamma-1}\rho^{\gamma-1}
+\frac{\delta\beta}{\beta-1}\rho^{\beta-1}\Big)\rho_B|u_B\cdot \nu(x)|\,d\sigma(x)dt
+3\int^{\tau}_0\int_{\Omega\times\mathbb{R}^3}f\,dxdvdt
\nonumber\\
&+\int^{\tau}_0\int_{\Omega}\big(-(\rho^{\gamma}+\delta\rho^{\beta}){\rm{div}}u_\infty
-\mathbb{S}(\nabla u_\infty):\nabla(u-u_\infty)
-\rho u\cdot\nabla u_\infty\cdot(u-u_\infty) \big)\,dxdt\nonumber\\
&-\int^{\tau}_0\int_{\Omega}(j-nu)\cdot u_\infty\,dxdt.
\end{align*}
\end{prop}

\subsection{Taking the limit $\delta\rightarrow 0$}
We want to pass to the limit $\delta\rightarrow 0$ and relax our hypotheses on the initial and boundary data \eqref{11}.
Following a very similar argument as that in subsection \ref{le}, we can take the limit on $\delta$ in the weak formulation of
Cucker-Smale system \eqref{54}.
Finally, the limit $\delta\to 0$ in the weak formulation of \eqref{56}-\eqref{57} can be carried out similarly as in \cite{CJN}, and therefore we omit it.
Hence the proof of Theorem \ref{main} is completed.
\hfill\qedsymbol

\medskip
\indent
{\bf Acknowledgements:}
The authors would like to thank Professor Fucai Li for his fruitful discussions and encouragement during the preparation of this paper.
Y. Li are supported by NSFC (Grant No. 12071212).
NZ acknowledges support from the Alexander von Humboldt Foundation (AvH), from the Austrian Science Foundation (FWF) grant P33010, and from the bilateral
Croatian-Austrian Project of the Austrian Agency for International Cooperation in Education and
Research (OAD) grant HR 19/2020.

\medskip\smallskip

\end{document}